\documentclass[a4paper,12pt]{amsart}
\usepackage{verbatim}
\usepackage[english]{babel}
\usepackage[utf8]{inputenc}
\usepackage{amsthm}
\usepackage[T1]{fontenc}
\usepackage{amsmath}
\usepackage{enumerate}
\usepackage{mathabx}
\usepackage{mathtools}
\usepackage{amssymb}
\usepackage{tikz}
\theoremstyle{plain}
\newtheorem{Teo}{Theorem}[section]

\newtheorem{Prop}[Teo]{Proposition}
\theoremstyle{definition}

\newtheorem{Def}[Teo]{Definition}
\numberwithin{equation}{section}

\DeclareMathOperator{\Lip}{Lip}

\DeclareMathOperator{\spanop}{span}

\title{On the duality of the\\ symmetric strong diameter $2$ property\\ in Lipschitz spaces}
\author{Andre Ostrak}
\date{}

\begin{document}

\begin{abstract}
We characterise the weak$^*$ symmetric strong diameter $2$ property in Lipschitz function spaces by a property of its predual, the Lipschitz-free space. We call this new property decomposable octahedrality and study its duality with the symmetric strong diameter $2$ property in general. For a Banach space to be decomposably octahedral it is sufficient that its dual space has the weak$^*$ symmetric strong diameter $2$ property. Whether it is also a necessary condition remains open.
\end{abstract}

\maketitle

\section{Introduction}
We consider only real nontrivial Banach spaces. We start by fixing some notation. Let $X$ be a Banach space. Denote its closed unit ball, unit sphere, and dual space by $B_X$, $S_X$, and $X^*$, respectively. A \emph{weak$^*$ slice} of $B_{X^*}$ is a set of the form
\[
S(B_{X^*},x, \alpha)\coloneqq \{x^*\in B_{X^*}\colon x^*(x)>1-\alpha\},
\]
where $x\in S_X$ and $\alpha>0$.

Let $M$ be a pointed metric space, that is, a metric space with a fixed point $0$. The space $\Lip_0(M)$ of all Lipschitz functions $f\colon M\to \mathbb{R}$ with $f(0)=0$ is a Banach space with the norm
\[
\|f\|_{\Lip} =\sup\left\{ \frac{|f(x)-f(y)|}{d(x,y)}\colon x,y\in M,\; x\neq y\right\}.
\]

Let $\delta_x\colon M\to \mathbb{R}$ be the characteristic function of the one element set $\{x\}$.
The Lipschitz-free space $\mathcal{F}(M)$ is defined as the completion of the molecule space
\[
\mathcal{M}(M)=\spanop\left\{\delta_p-\delta_q \colon p,q\in M \right\},
\]
equipped with the Arens--Eells norm
\[
\|\mu\|=\inf\Big\{\sum_{i=1}^n|\lambda_i| d(p_i, q_i)\colon \mu=\sum_{i=1}^n\lambda_i(\delta_{p_i}-\delta_{q_i}),\; p_i, q_i\in M, \;n\in \mathbb{N}\Big\},
\]
where the infimum is taken over all expressions of molecule $\mu\in \mathcal{M}(M)$ as a linear combination of
elementary molecules $\delta_p-\delta_q$ (see details in \cite{W}). For any $p,q\in M$,
\[
\|\delta_p-\delta_q\|=d(p,q).
\]
It can be shown that $\mathcal{F}(M)^*$ is isometrically isomorphic to $\Lip_0(M)$, where the isomorphism can be defined as  $T\colon \Lip_0(M)\to \mathcal{M}(M)^*$,
\[
(Tf)(\mu)=\sum_{p\in M} f(p)\mu(p),
\]
where $f\in \Lip_0(M)$, $\mu\in \mathcal{M}(M)$.

Recall that the dual Banach space $X^*$ has the \emph{weak$^*$ strong diameter $2$ property} ($w^*$-SD$2$P) if every finite convex combination of weak$^*$ slices of $B_{X^*}$ has diameter $2$. The (norm of) Banach space $X$ is said to be octahedral (OH) if, for any $x_1,\ldots, x_n\in S_X$ and $\varepsilon>0$, there exists a $y\in S_X$ such that, for any $i\in \{1,\ldots, n\}$,
\[
\|x_i+y\|\geq 2-\varepsilon.
\]

It is well known that the dual space $X^*$ has the $w^*$-SD$2$P if and only if the norm of Banach space $X$ is octahedral (\cite{D},\cite{G}, for a proof, see, e.g., \cite{BLR} or \cite{HLP}). This means that the Lipschitz space $\Lip_0(M)$ has the $w^*$-SD$2$P if and only if the norm of the Lipschitz-free space $\mathcal{F}(M)$ is octahedral. In \cite{PR}, it was shown that the octahedrality of $\mathcal{F}(M)$ can also be characterised by the following property of the metric space $M$.
\begin{Def}
A metric space $M$ is said to have the \emph{long trapezoid property} (LTP) if, for every finite subset $N$ of $M$ and $\varepsilon>0$, there exist $u,v\in M$, $u\neq v$, such that, for any $x,y\in N$,
\begin{equation*}
(1-\varepsilon)\bigl(d(x,y)+d(u,v)\bigr)\leq d(x,u)+d(y,v).
\end{equation*}
\end{Def}
\noindent
More precisely, it was shown that the following theorem holds. 
\begin{Teo}{\cite[Theorem 3.1]{PR}}\label{Teo: LTP} Let $M$ be a pointed metric space. The following statements are equivalent:
\begin{enumerate}[\upshape (i)]
    \item $\Lip_0(M)$ has the $w^*$-SD$2$P;
    \item the norm of $\mathcal{F}(M)$ is OH;
    \item $M$ has the LTP.
\end{enumerate}
\end{Teo}

The objective of this paper is to give a similar characterisation to the following property, which was introduced only recently but has already been under rigorous study (see \cite{ALN}, \cite{ANP}, \cite{CCGMR}, \cite{HLLN}, \cite{L}, and \cite{LR}).
\begin{Def}
A dual Banach space $X^*$ is said to have the \emph{weak$^*$ symmetric strong diameter $2$ property} ($w^*$-SSD$2$P) if, for every finite family $\{S_i\}_{i=1}^n$ of weak$^*$ slices of $B_{X^*}$ and $\varepsilon>0$, there exist $f_i\in S_i$, $i=1,\ldots,n$, and $g\in B_{X^*}$ such that $f_i\pm g\in S_i$ for every $i\in \{1,\ldots,n\}$ and $\|g\|>1-\varepsilon$.
\end{Def}

It is known that generally the $w^*$-SSD$2$P is a strictly stronger property than the $w^*$-SD$2$P. In fact, a Lipschitz function space with the $w^*$-SD$2$P but without the $w^*$-SSD$2$P appeared in \cite{O}. Moreover, it was shown in \cite{O} that the Lipschitz space $\Lip_0(M)$ has the $w^*$-SSD$2$P if and only if the metric space $M$ has the following property.
\begin{Def}
A metric space $M$ is said to have the \emph{strong long trapezoid property} (SLTP) if, for every finite subset $N$ of $M$ and $\varepsilon>0$, there exist $u,v\in M$,  $u\neq v$, such that, for any $x,y \in N$,
\begin{equation}\label{eq: LTP}
(1-\varepsilon)\bigl(d(x,y)+d(u,v)\bigr)\leq d(x,u)+d(y,v),
\end{equation}
and, for any $x,y,z,w\in N$,
\begin{equation}\label{eq: SLTP}
\begin{aligned}
(1-\varepsilon)&\bigl(d(x,y)+d(z,w)+2d(u,v)\bigr)\\
&\qquad\qquad\leq d(x,u)+d(y,u)+d(z,v)+d(w,v).
\end{aligned}
\end{equation}
\end{Def}

We now introduce the property that, via Lipschitz-free space $\mathcal{F}(M)$, characterises the $w^*$-SSD$2$P of the Lipschitz space $\Lip_0(M)$.

\begin{Def}
We say that the (norm of) Banach space $X$ is \emph{decomposably octahedral} (DOH) if, for every finite subset $E$ of $S_X$ and $\varepsilon>0$, there exists a $y\in S_X$ such that, for any $y_1,\ldots, y_n \in X$ with $\sum_{i=1}^n y_i=y$, and, for any $x_1,\ldots, x_n\in E$, $a_1, b_1,\ldots,a_n, b_n \geq 0$, the following inequality holds
\begin{equation*}
\begin{aligned}
\sum_{i=1}^n \bigr(\|a_i x_i +y_i\|+\|b_i x_i -y_i\|\bigr)\geq (1-\varepsilon)\Big(\sum_{i=1}^n (a_i+b_i)+2\Big).
\end{aligned}
\end{equation*}
\end{Def}
It is easy to verify that OH follows from DOH. The main objective of this paper is to show that the Lipschitz space $\Lip_0(M)$ has the $w^*$-SSD$2$P if and only if $\mathcal{F}(M)$ is DOH. More generally, we show that if the dual space $X^*$ has the $w^*$-SSD$2$P then $X$ is DOH. Whether the converse is true, is currently unknown to us. We finish the paper by looking through examples of octahedral Banach spaces whose duals are known not to have the $w^*$-SSD2P. These Banach spaces also fail to be DOH.

The paper is organised as follows.

In Section \ref{sec: 2}, we show that if the dual Banach space $X^*$ has the $w^*$-SSD$2$P then $X$ is DOH. In addition, we prove Theorem \ref{main}, which says that $\Lip_0(M)$ has the $w^*$-SSD$2$P if and only if $\mathcal{F}(M)$ is DOH.

In Section \ref{sec: 3}, we prove Proposition \ref{prop: DOH direct sums}, which gives necessary and sufficient conditions for the absolute sum of two Banach spaces to be DOH. We finish the paper by showing that the space $C[0,1]$, the norm of which is known to be octahedral, is not DOH.

\section{Main results}\label{sec: 2}
In this section, we show that if a dual Banach space $X^*$ has the $w^*$-SSD$2$P then $X$ is DOH. Whether the reverse implication holds in general, is unknown to us. However, in the following, we prove that the reverse implication holds if $X$ is a Lipschitz-free space.
\begin{Prop}\label{Prop}
Let $X$ be a Banach space. If $X^*$ has the $w^*$-SSD$2$P then $X$ is DOH. 
\end{Prop}

\begin{proof}
Assume that $X^*$ has the $w^*$-SSD$2$P. Let $E$ be a finite subset of $S_X$ and $\varepsilon>0$.
For any $x\in E$, define a $w^*$-slice $S_{x} = S\left(B_{X^*}, x, \frac{\varepsilon}{2}\right)$.
Since $\Lip_0(M)$ has the $w^*$-SSD$2$P, we can find $f_{x}\in S_{x}$ and $g\in B_{X^*}$ such that $\|f_{x}\pm g\|\leq 1$ for every $x\in E$ and $\|g\|\geq 1-\varepsilon$. Then, for any $x\in E$,
\[
f_{x}(x)\geq 1-\frac{\varepsilon}{2}\quad \text{and}\quad |g(x)|\leq \frac{\varepsilon}{2}.
\]
Let $y\in B_X$ be such that $g(y)\geq 1-\varepsilon$. For any $y_1,\ldots, y_n \in X$ with $\sum_{i=1}^n y_i=y$, and, for any $x_1,\ldots,x_n\in E$, $a_1, b_1,\ldots,a_n, b_n\geq 0$, we have
\begin{align*}
&\sum_{i=1}^n \bigr(\|a_i x_i +y_i\|+\|b_i x_i -y_i\|\bigr)\\
&\qquad\qquad \geq \sum_{i=1}^n \bigr((f_{x_i}+g)(a_i x_i+y_i)+ (f_{x_i}-g)(b_i x_i-y_i)\bigr)\\
&\qquad\qquad=\sum_{i=1}^n \bigr((a_i+b_i) f_{x_i}(x_i)+(a_i-b_i) g(x_i)+2g(y_i)\bigr)\\
&\qquad\qquad\geq (1-\varepsilon)\sum_{i=1}^n (a_i+b_i)+2\sum_{i=1}^n g(y_i)\\
&\qquad\qquad=(1-\varepsilon)\sum_{i=1}^n (a_i+b_i)+2g(y)\\
&\qquad\qquad\geq(1-\varepsilon)\Big(\sum_{i=1}^n (a_i+b_i)+2\Big).
\end{align*}
Therefore, $X$ is DOH.
\end{proof}

\begin{Teo}[cf. {\cite[Theorem 2.1]{O}}]\label{main}
Let $M$ be a pointed metric space. The following statements are equivalent:
\begin{enumerate}[\upshape (i)]
    \item $\Lip_0(M)$ has the $w^*$-SSD$2$P;
    \item $\mathcal{F}(M)$ is DOH;
    \item $M$ has the SLTP.
\end{enumerate}
\end{Teo}

\begin{proof}
(i)$\Leftrightarrow$(iii) is \cite[Theorem 2.1]{O}.

(i)$\Rightarrow$(ii) holds by Proposition \ref{Prop}.

(ii)$\Rightarrow$(iii). Assume that the Lipschitz-free space $\mathcal{F}(M)$ is DOH. Let $N$ be a finite subset of $M$ and $0<\varepsilon<\frac{1}{2}$. Define $E = \big\{\frac{\delta_p-\delta_q}{d(p,q)}\colon p,q\in N\big\}$ and let $0<\delta<\frac{r\varepsilon}{2R}$, where $r, R>0$ are such that
\[
r<d(p,q)<R\qquad \text{for any $p,q\in N$, $p\neq q$.}
\]
Since $\mathcal{F}(M)$ is DOH, there exists a $\nu\in \spanop\{\delta_p-\delta_q\colon p,q\in M\}$, $\|\nu\|<1$, such that, for any $\nu_1,\ldots, \nu_n \in \mathcal{F}(M)$ with $\sum_{i=1}^n \nu_i=\nu$, and, for any $\mu_1,\ldots, \mu_n\in E$, $a_1, b_1,\ldots, a_n, b_n\geq 0$, the following inequality holds
\begin{equation*}
\begin{aligned}
\sum_{i=1}^n \bigr(\|a_i \mu_i +\nu_i\|+\|b_i \mu_i -\nu_i\|\bigr)\geq (1-\delta)\Big(\sum_{i=1}^n (a_i+b_i)+2\Big).
\end{aligned}
\end{equation*}

Since
\[
\|\nu\|=\inf\Big\{\sum_{i=1}^n|\lambda_i| d(p_i, q_i)\colon \nu=\sum_{i=1}^n\lambda_i(\delta_{p_i}-\delta_{q_i}),\; p_i, q_i\in M\Big\},
\]
there exist $n\in \mathbb{N}$ and $\lambda_i>0$, $u_i, v_i\in M$, $u_i\neq v_i$, $i=1,\ldots, n$, such that $\nu=\sum_{i=1}^n \lambda_i (\delta_{u_i}-\delta_{v_i})$ and $\sum_{i=1}^n \lambda_i d(u_i, v_i)= 1$.

It suffices to show that there exists an $i\in \{1,\ldots, n\}$ such that, taking $u = u_i$ and $v = v_i$, the inequalities \eqref{eq: LTP} and \eqref{eq: SLTP} hold for any $x,y,z,w\in N$. Suppose that, contrary to our claim, for any $i\in \{1,\ldots, n\}$, there exist $x_i, y_i\in N$ such that
\begin{equation}\label{neq: LTP}
(1-\varepsilon)\bigl(d(x_i,y_i)+d(u_i,v_i)\bigr)>d(x_i,u_i)+d(y_i,v_i),
\end{equation}
or $x_i, y_i, z_i, w_i\in N$ such that
\begin{equation}\label{neq: SLTP}
\begin{aligned}
&(1-\varepsilon)\bigl(d(x_i,y_i)+d(z_i,w_i)+2d(u_i,v_i)\bigr)\\
&\qquad\qquad\qquad> d(x_i,u_i)+d(y_i,u_i)+d(z_i,v_i)+d(w_i,v_i).
\end{aligned}
\end{equation}
Let $I$ be the subset of indexes $\{1,\ldots, n\}$ for which there exist $x_i, y_i\in N$ such that \eqref{neq: LTP} holds, and let $J$ be the set $\{1,\ldots, n\}\setminus I$. By our assumption, for every $i\in J$, there exist $x_i, y_i, z_i, w_i\in N$ such that $x_i\neq y_i$ or $z_i\neq w_i$, and \eqref{neq: SLTP} holds.
Fix such $x_i, y_i\in N$ for every $i\in I$, and $x_i, y_i, z_i, w_i$ for every $i\in J$. Then
\begin{align*}
&\sum_{i\in I} \lambda_i\Bigr(d(x_i,u_i)+d(y_i,v_i)-(1-\varepsilon)\bigl(d(x_i,y_i)+d(u_i,v_i)\bigr)\Bigr)\\
&\qquad +\sum_{i\in J}\lambda_i\bigr( d(x_i,u_i)+d(y_i,u_i)+d(z_i,v_i)+d(w_i,v_i)\bigr)\\
&\qquad -(1-\varepsilon)\sum_{i\in J}\lambda_i\bigr(d(x_i,y_i)+d(z_i,w_i)+2d(u_i,v_i)\bigr)<0.
\end{align*}

To prove that this can not be the case, we show that the following inequality holds
\begin{equation*}\label{ineq: eps-delta}
\begin{aligned}
&\sum_{i\in I} \lambda_i\Big(d(x_i,u_i)+d(y_i,v_i)-(1-\varepsilon)\bigr(d(x_i,y_i)+d(u_i, v_i)\bigr)\Big)\\
&\qquad\qquad +\sum_{i\in J} \lambda_i\bigr(d(x_i,u_i)+d(y_i,u_i)+d(z_i,v_i)+d(w_i,v_i)\bigr)\\
&\qquad\qquad -(1-\varepsilon)\sum_{i\in J} \lambda_i\bigr(d(x_i,y_i)+d(z_i, w_i)+2d(u_i, v_i)\bigr)\\
&\qquad \geq \sum_{i\in I} \lambda_i\bigr(d(x_i,u_i)+d(y_i,v_i)+d(u_i,v_i)\bigr)\\
&\qquad\qquad-(1-\delta)\sum_{i\in I} \lambda_i\bigr(d(x_i,y_i)+2d(u_i, v_i)\bigr)\\
&\qquad \qquad +\sum_{i\in J} \lambda_i\bigr(d(x_i,u_i)+d(y_i,u_i)+d(z_i,v_i)+d(w_i,v_i)+\delta d(y_i, z_i)\bigr)\\
&\qquad\qquad -(1-\delta)\sum_{i\in J} \lambda_i\bigr(d(x_i,y_i)+d(z_i, w_i)+2d(u_i, v_i)\bigr)
\end{aligned}
\end{equation*}
and that the right hand side of this inequality is nonnegative. 

The inequality holds because, since $2\delta\leq \varepsilon$, for any $i\in I$, we have
\begin{align*}
&d(x_i,u_i)+d(y_i,v_i)-(1-\varepsilon)\bigr(d(x_i,y_i)+d(u_i, v_i)\bigr)\\
&\qquad\geq d(x_i,u_i)+d(y_i,v_i)-(1-2\delta)\bigr(d(x_i,y_i)+d(u_i, v_i)\bigr)\\
&\qquad\geq d(x_i,u_i)+d(y_i,v_i)+d(u_i,v_i)-(1-\delta)\bigr(d(x_i,y_i)+2d(u_i, v_i)\bigr),\\
\end{align*}
and, for any $i\in J$, since $\delta d(y_i, z_i)\leq\frac{\varepsilon}{2}\bigr(d(x_i, y_i)+d(z_i, w_i)\bigr)$, we have
\begin{align*}
&d(x_i,u_i)+d(y_i,u_i)+d(z_i,v_i)+d(w_i,v_i)\\
&\qquad\qquad -(1-\varepsilon) \bigr(d(x_i,y_i)+d(z_i, w_i)+2d(u_i, v_i)\bigr)\\
&\qquad\geq  d(x_i,u_i)+d(y_i,u_i)+d(z_i,v_i)+d(w_i,v_i)\\
&\qquad\qquad -(1-\delta-\frac{\varepsilon}{2})\bigr(d(x_i,y_i)+d(z_i, w_i)+2d(u_i, v_i)\bigr)\\
&\qquad\geq d(x_i,u_i)+d(y_i,u_i)+d(z_i,v_i)+d(w_i,v_i)+\delta d(y_i, z_i)\\
&\qquad\qquad -(1-\delta) \bigr(d(x_i,y_i)+d(z_i, w_i)+2d(u_i, v_i)\bigr).
\end{align*}

It remains to prove that 
\begin{align*}
&\sum_{i\in I} \lambda_i\Big(d(x_i,u_i)+d(y_i,v_i)+d(u_i,v_i)-(1-\delta)\bigr(d(x_i,y_i)+2d(u_i, v_i)\bigr)\Big)\\
&\qquad \qquad +\sum_{i\in J} \lambda_i\bigr(d(x_i,u_i)+d(y_i,u_i)+d(z_i,v_i)+d(w_i,v_i)+\delta d(y_i, z_i)\bigr)\\
&\qquad\qquad -(1-\delta)\sum_{i\in J} \lambda_i\bigr(d(x_i,y_i)+d(z_i, w_i)+2d(u_i, v_i)\bigr)\geq 0.
\end{align*}
To this end, note that
\begin{align*}
&\sum_{i\in I} \lambda_i\bigr(d(x_i,u_i)+d(y_i,v_i)+d(u_i,v_i)\bigr)\\
&\qquad\qquad +\sum_{i\in J} \lambda_i\bigr(d(x_i, u_i)+d(y_i, u_i)+d(y_i, z_i)+d(z_i, v_i)+d(w_i, v_i)\bigr)\\
&\qquad \geq\sum_{i\in I} \lambda_i\bigr(\|\delta_{x_i}-\delta_{y_i}-(\delta_{u_i}-\delta_{v_i})\|+\|\delta_{u_i}-\delta_{v_i}\|\bigr)\\
&\qquad\qquad +\sum_{i\in J} \lambda_i\bigr(\|\delta_{x_i}-\delta_{y_i}-(\delta_{u_i}-\delta_{y_i})\|+\|\delta_{u_i}-\delta_{y_i}\|\bigr)\\
&\qquad\qquad +\sum_{i\in J}\lambda_i\bigr(\|\delta_{y_i}-\delta_{z_i}-(\delta_{y_i}-\delta_{z_i})\|+\|\delta_{y_i}-\delta_{z_i}\|\bigr)\\
&\qquad\qquad +\sum_{i\in J}\lambda_i\bigr(\|\delta_{z_i}-\delta_{w_i}-(\delta_{z_i}-\delta_{v_i})\|+\|\delta_{z_i}-\delta_{v_i}\| \bigr)\\
&\qquad\geq 2(1-\delta)+(1-\delta)\sum_{i\in I} \lambda_id(x_i,y_i)\\
&\qquad \qquad+(1-\delta)\sum_{i\in J}\lambda_i \bigr(d(x_i, y_i)+d(y_i,z_i)+d(z_i,w_i)\bigr)\\
&\qquad= (1-\delta)\sum_{i\in I} \lambda_i\bigr(d(x_i,y_i)+2d(u_i, v_i)\bigr)\\
&\qquad \qquad+(1-\delta)\sum_{i\in J} \lambda_i\bigr(d(x_i, y_i)+d(y_i,z_i)+d(z_i,w_i)+2d(u_i, v_i)\bigr),
\end{align*}
where the second inequality holds by our choice of $\nu$ because 
\[
\nu = \sum_{i\in I} (\delta_{u_i}-\delta_{v_i})+\sum_{i\in J} (\delta_{u_i}-\delta_{y_i}+\delta_{y_i}-\delta_{w_i}+\delta_{w_i}-\delta_{v_i}).
\]
This completes the proof.
\end{proof}

\section{Decomposable octahedrality in Banach spaces}\label{sec: 3}
In this section, we look at examples of octahedral Banach spaces for which it is known that the dual space does not have the $w^*$-SSD$2$P. These Banach spaces also fail to be decomposably octahedral. This leaves open the question of whether the reverse implication of Proposition \ref{Prop} holds.

We start by looking at decomposable octahedrality in absolute sums of Banach spaces. Recall that a norm $N$ on $\mathbb{R}^2$ is \emph{absolute} if
\[
N(a,b)=N(|a|,|b|)\qquad \text{for all $(a,b)\in \mathbb{R}^2$,}
\]
and \emph{normalised} if
\[
N(1,0)=N(0,1)=1.
\]

For $1\leq p\leq \infty$, we denote the $\ell_p$ norm on $\mathbb{R}^2$ by $\|\cdot\|_p$. Every $\ell_p$ norm is an absolute normalised norm.

For Banach spaces $X,Y$, we denote by $X\oplus_N Y$ the product space $X\times Y$ equipped with the norm $N$, where
\[
N(x,y)=N(\|x\|,\|y\|)\qquad \text{for all $x\in X, y\in Y$.}
\]
In case $N$ is an $\ell_p$ norm we write $X\oplus_p Y$. 

It can be shown that $X^*\oplus_N Y^*$ has the $w^*$-SSD$2$P if and only if $N$ is the $\ell_\infty$ norm and $X^*$ or $Y^*$ has the $w^*$-SSD$2$P (the proof is similar to the one of \cite[Theorem 3.1]{HLLN}). 

We give necessary and sufficient conditions for the absolute sum of Banach spaces to be DOH.

\begin{Prop}\label{prop: DOH direct sums}
Let $X,Y$ be Banach spaces.
\begin{enumerate}[(a)]
\item The space $X\oplus_1 Y$ is DOH if and only if $X$ or $Y$ is DOH.
\item If $N$ is an absolute normalised norm different from the $\ell_1$ norm then the space $X\oplus_N Y$ is not DOH.
\end{enumerate}
\end{Prop}

\begin{proof}
(a). First, assume that $X$ is DOH. Let $E$ be a finite subset of $S_{X\oplus_1 Y}$ and $\varepsilon>0$. Since $X$ is DOH, there exists a $z\in S_X$ such that, for any $z_1,\ldots, z_n\in X$ with $\sum_{i=1}^n z_i=z$, and, for any $a_1,b_1,\ldots, a_n, b_n\geq 0$, $(x_1, y_1),\ldots, (x_n, y_n)\in E$, we have
\begin{align*}
\sum_{i=1}^n \bigr(\|a_i x_i +z_i\|+\|b_i x_i -z_i\|\bigr)\geq (1-\varepsilon)\Big(\sum_{i=1}^n (a_i+b_i)\|x_i\|+2\Big).
\end{align*}
Notice that $N(z,0)=1$. Let $(z_1,w_1),\ldots, (z_n, w_n)\in X\oplus_1 Y$ be such that $\sum_{i=1}^n (z_i, w_i)=(z,0)$. Then, for any $a_1,b_1,\ldots,a_n, b_n\geq 0$,\linebreak $(x_1, y_1),\ldots,(x_n, y_n)\in E$, we have
\begin{align*}
&\sum_{i=1}^n \bigr(\|a_i (x_i, y_i)+(z_i, w_i)\|+\|b_i(x_i, y_i)-(z_i, w_i)\|\bigr)\\
&\qquad = \sum_{i=1}^n \bigr(\|a_i x_i +z_i\|+\|a_i y_i+w_i\|+\|b_i x_i-z_i\|+\|b_i y_i -w_i\|\bigr)\\
&\qquad \geq (1-\varepsilon)\Big(\sum_{i=1}^n (a_i+b_i)\|x_i\|+2\Big)+\sum_{i=1}^n (a_i+b_i) \|y_i\|\\
&\qquad \geq (1-\varepsilon)\Big(\sum_{i=1}^n (a_i+b_i)+2\Big).
\end{align*} 
Therefore, $X\oplus_1 Y$ is DOH.

Assume now that $X$, $Y$ are not DOH. Then there exist finite subsets $E_1$, $E_2$ of $S_X$ and $S_Y$, respectively, and $\varepsilon>0$, such that, for any $z\in S_X$, $w\in S_Y$, there exist $n\in \mathbb{N}$ and $z_i\in X$, $w_i\in Y$, $a_i,b_i,c_i, d_i\geq 0$, $x_i\in E_1$, $y_i\in E_2$, $i=1,\ldots,n$, such that $\sum_{i=1}^n z_i=z$, $\sum_{i=1}^n w_i=w$,
\begin{align*}
\sum_{i=1}^n \bigr(\|a_i x_i +z_i\|+\|b_i x_i -z_i\|\bigr)< (1-\varepsilon)\Big(\sum_{i=1}^n (a_i+b_i)+2\Big),
\end{align*}
and
\begin{align*}
\sum_{i=1}^n \bigr(\|c_i y_i +w_i\|+\|d_i y_i -w_i\|\bigr)< (1-\varepsilon)\Big(\sum_{i=1}^n (c_i+d_i)+2\Big).
\end{align*}

Now, take $E=\{(x,0), (0,y)\colon x\in E_1, y\in E_2\}$. This is a finite subset of $S_{X\oplus_1 Y}$. However, for any $(\overline{z}, \overline{w})\in S_{X\oplus_1 Y}$, there exist $n\in \mathbb{N}$ and $\overline{z}_i\in X$, $\overline{w}_i\in Y$, $a_i,b_i,c_i,d_i\geq 0$, $(x_i, 0), (0,y_i)\in E$, $i=1,\ldots, n$, such that $\sum_{i=1}^n \overline{z}_i=\overline{z}$, $\sum_{i=1}^n \overline{w}_i=\overline{w}$, and
\begin{align*}
&\sum_{i=1}^n \bigr(\bigr\|\|\overline{z}\|a_i (x_i,0) +(\overline{z}_i,0)\bigr\|+\bigr\|\|\overline{z}\|b_i (x_i,0) -(\overline{z}_i,0)\bigr\|\bigr)\\
&\qquad\qquad+\sum_{i=1}^n \bigr(\bigr\|\|\overline{w}\|c_i (0,y_i) +(0,\overline{w}_i)\bigr\|+\bigr\|\|\overline{w}\|d_i (0,y_i) -(0,\overline{w}_i)\bigr\|\bigr)\\
&\qquad=\sum_{i=1}^n \bigr(\bigr\|\|\overline{z}\|a_i x_i +\overline{z}_i\bigr\|+\bigr\|\|\overline{z}\|b_i x_i -\overline{z}_i\bigr\|\bigr)\\
&\qquad\qquad+\sum_{i=1}^n \bigr(\bigr\|\|\overline{w}\|c_i y_i +\overline{w}_i\bigr\|+\bigr\|\|\overline{w}\|d_i y_i -\overline{w}_i\bigr\|\bigr)\\
&\qquad< (1-\varepsilon) \bigr(\|\overline{z}\|(a_i+b_i)+ \|\overline{w}\|(c_i+d_i)+2\bigr).
\end{align*}
Therefore, $X\oplus_1 Y$ is not DOH.

(b). Take $x\in S_X$ and $w\in S_Y$. Then $(x,0),(0,w)\in S_{X\oplus_N Y}$. Notice that, for any $y\in S_{X\oplus_N Y}$, there exist $y_1\in B_X$, $y_2\in B_Y$ such that $y=(y_1, 0)+(0,y_2)$. Since $N$ is not the $\ell_1$ norm, there exists an $\varepsilon>0$ such that $N(1,1)<2(1-2\varepsilon)$. Thus,
\begin{align*}
&\bigr\|\|y_1\|(x,0)+(0,y_2)\bigr\|+\bigr\|\|y_2\|(x,0)-(0, y_2)\bigr\|\\
&\qquad\qquad+\bigr\|\|y_1\|(0,w)+(y_1,0)\bigr\|+\bigr\|\|y_2\|(0,w)-(y_1,0)\bigr\|\\
&\qquad=N(\|y_1\|,\|y_2\|)+N(\|y_2\|,\|y_2\|)+N(\|y_1\|,\|y_1\|)+N(\|y_2\|,\|y_1\|)\\
&\qquad=2+\|y_2\|N(1,1)+\|y_1\|N(1,1)\\
&\qquad<2+2(1-2\varepsilon)\bigr(\|y_1\|+\|y_2\|\bigr)\\
&\qquad\leq (1-\varepsilon)\bigr(2\|y_1\|+2\|y_2\|+2\bigr).
\end{align*}
Therefore, $X\oplus_N Y$ is not DOH.
\end{proof}

Note that, by Proposition \ref{prop: DOH direct sums}, the space $\ell_1\oplus_\infty\ell_1$, which is known to be OH (see \cite{HLP}), is not 
DOH.

We finish the paper by noting that the space $C[0,1]$, which is OH (see \cite[Example 1.1]{HLP}), is not DOH. To see this, define $f_1, f_2\in S_{C[0,1]}$,
\[
f_1(x)=\left\{\begin{array}{ll}
1-4x, &\text{if $x\in \big[0,\frac{1}{4}\big]$,}\\
0, & \text{else;}
\end{array}\right.
\]
and
\[
f_2(x)=\left\{\begin{array}{ll}
4x-3, &\text{if $x\in \big[\frac{3}{4}, 1\big]$,}\\
0, &\text{else.}
\end{array}\right.
\]
Let $0<\varepsilon<\frac{1}{3}$. For any $g\in S_{C[0,1]}$, we can find $g_1,g_2\in B_{C[0,1]}$ such that $g=g_1+g_2$,
\[
g_1(x)=0\qquad \text{for all $x\in \Big[0, \frac{1}{4}\Big]$,}
\]
and
\[
g_2(x)=0\qquad \text{for all $x\in \Big[\frac{3}{4},1\Big]$.}
\] 
Then
\begin{align*}
\|f_1+g_1\|+\|f_1-g_1\|+\|f_2+g_2\|+\|f_2-g_2\|=4<(4+2)(1-\varepsilon).
\end{align*}
Therefore, $C[0,1]$ is not DOH.

\section*{Acknowledgements}
The paper is a part of a Ph.D. thesis which is being prepared
by the author at University of Tartu under the supervision of Rainis Haller and Märt Põldvere.
The author is grateful to his supervisors for their valuable help. This work was supported by the Estonian Research Council grant (PRG877).

\addcontentsline{toc}{section}{References}

\end{document}